\documentclass{daj}

\dajAUTHORdetails{%
  title = {Rank bounds for design matrices with block entries and geometric applications}, 
  author = {Zeev Dvir, Ankit Garg, Rafael Oliveira, and J\'{o}zsef Solymosi},
  plaintextauthor = {Zeev Dvir, Ankit Garg, Rafael Oliveira, and Jozsef Solymosi},
  plaintexttitle = {Rank bounds for design matrices with block entries and geometric applications}, 
  runningtitle = {Rank bounds for design matrices with block entries}, 
}   

\dajEDITORdetails{%
   year={2018},
   number={5},
   received={23 February 2017},   
   published={12 March 2018},  
   doi={10.19086/da.3118},       
}   

\usepackage{amsmath,amsthm,amssymb,hyperref, color,fullpage}
\newcommand{\remove}[1]{}

\newtheorem{thm}{Theorem}[section]
\newtheorem{claim}[thm]{Claim}
\newtheorem{lem}[thm]{Lemma}
\newtheorem{define}[thm]{Definition}
\newtheorem{cor}[thm]{Corollary}

\newtheorem{comm}[thm]{Comment}

\def\R{{\mathbb{R}}}

\def\C{{\mathbb{C}}}

\def\cK{{\cal K}}

\def\cM{\mathcal{M}}

\def\bp{{\mathbf p}}
\def\bq{{\mathbf q}}

\def\M{{\mathcal{M}}}

\def\_{\,\,\,\,\,}

\def\tr{\textsf{tr}}

\def\rank{\textsf{rank}}

\def\capac{\textsf{cap}}
\def\row{\textsf{Row}}
\def\col{\textsf{Col}}
\def\ds{\textsf{ds}}
\def\spn{\textsf{span}}

\newcommand{\eps}{\epsilon}

\begin{document}

\begin{frontmatter}[classification=text]

\title{Rank bounds for design matrices with block entries and geometric applications} 

\author[dvir]{Zeev Dvir\thanks{Research supported by NSF CAREER award DMS-1451191 and NSF grant CCF-1523816.}}
\author[garg]{Ankit Garg}
\author[oliveira]{Rafael Oliveira}
\author[solymosi]{J\'{o}zsef Solymosi\thanks{Research  partially supported by NSERC,  ERC Advanced Research Grant AdG. 321104 and by Hungarian National Research Grant NK 104183.}}

\begin{abstract}
Design matrices are sparse matrices in which the supports of different columns intersect in a few positions. Such matrices come up naturally when studying problems  involving point sets with many collinear triples. In this work we consider  design matrices with block (or matrix)  entries. Our main result is a lower bound on the rank of such matrices, extending the bounds proven in \cite{BDWY12,DSW12} for the scalar case. As a result we obtain several applications in combinatorial geometry.  The first application involves extending the notion of structural rigidity (or graph rigidity) to the setting where we wish to bound the number of `degrees of freedom' in perturbing a set of points under  collinearity constraints (keeping some family of triples collinear). Other applications are an asymptotically tight  Sylvester-Gallai type result for arrangements of subspaces (improving \cite{DH}) and a new incidence bound for high dimensional line/curve arrangements. 

The main technical tool in the proof of the rank bound is an extension of the technique of matrix scaling to the setting of block matrices. We generalize the definition of doubly stochastic matrices to matrices with block entries and derive sufficient conditions for a doubly stochastic scaling to exist. 
\end{abstract}
\end{frontmatter}

\section{Introduction}
Design matrices, defined in \cite{BDWY12}, are (complex) matrices that satisfy certain conditions on their support (the set of non-zero entries). Roughly speaking, a design matrix has few non-zero entries per row, many non-zero entries per column and, most importantly, the supports of every  two columns intersect in a small number of positions. In \cite{BDWY12,DSW12}, lower bounds on the rank of such matrices were given and applied to upper bound the dimension of point configurations in $\C^d$ containing many collinear triples. In particular, \cite{DSW12} used this method to give a new elementary proof of Kelly's theorem (the complex analog of the Sylvester-Gallai theorem). In this work we generalize the rank bounds from \cite{BDWY12,DSW12} to handle design matrices with matrix entries. We then use these bounds to prove several new results in combinatorial geometry. 

Our geometric applications are of three types. The first deals with bounding the number of `degrees of freedom' when smoothly perturbing a set of points while maintaining a certain family of triples collinear. This is in the same spirit of structural rigidity results \cite{laman1970} in which pairwise distances are maintained along the edges of a graph embedded in the plane. The second application is a generalization of the Sylvester-Gallai theorem for arrangements of subspaces. Such a result was recently proved in \cite{DH} and we are able to give an asymptotically tight improvement to their results. The last application involves arrangements of lines and curves in $\C^d$ that have many pairwise incidences (each line/curve intersects many others). We are able to show upper bounds on the dimension of such configurations as a function of the number of incidences and under the assumption that no low dimensional subspace contains `too many' of the lines/curves.
	
The main tool used  to prove the rank bounds for design matrices in \cite{BDWY12,DSW12} was matrix scaling. Given a complex matrix $A = (A_{ij})$, we try to find coefficients $r_i,c_j$ for each row/column so that the matrix with entries $B_{ij} = r_iA_{ij}c_j$ is doubly stochastic. In this setting, one is actually interested  in the $\ell_2$ norms of all rows/columns being equal (instead of $\ell_1$). The main technical difficulty is in giving sufficient conditions that guarantee the existence of such a scaling. Following the pioneering work of Sinkhorn \cite{Sink}, such conditions are analyzed completely in \cite{RS}. To handle design matrices with block entries we study the problem of matrix scaling for block matrices. Finding sufficient conditions for scaling is intimately related to the well studied problem of operator scaling \cite{gurvits2004,LSW,GGOW15}. We give a (mostly) self-contained and elementary derivation of sufficient conditions for scaling to exist relying only on the work of \cite{BCCT} which gives sufficient conditions for scaling of matrices with one column (see Theorem~\ref{thm-BCCT} below). We note that \cite{BCCT} does not mention matrix scaling explicitly in their work (which studies the Brascamp-Lieb inequalities). The observation that this part of their work can be interpreted through this angle seems to not have been noticed before.

 We describe our results in more detail in the subsections below.  The main technical work involving matrix scaling will be discussed in Section~\ref{sec-scalingcapacity}.

\subsection{Design matrices with block entries}
For the rest of the paper, all matrices are complex unless otherwise noted. By positive definite (semi-definite) matrix we mean Hermitian matrix with positive (non-negative) eigenvalues.

	Let $\M_{m,n}(r,c)$ denote the set of $m \times n$ matrices with entries being $r \times c$ matrices. When referring to rows (and columns) of $A$ we mean the $m$ rows of blocks (and $n$ columns). We sometimes refer to the entries of $A$ as the blocks of $A$. For a matrix $A \in \M_{m,n}(r,c)$ we denote by $\tilde A$ the $\M_{rm,cn}(1,1)$ matrix obtained from $A$ in the natural way (ignoring blocks). We define  $\rank(A)$ to be the rank of $\tilde A$ (as a complex matrix). We will sometimes identify a matrix $A \in \M_{m,n}(r,c)$  with a linear map from $\C^{nc}$ to $\C^{mr}$ given by $\tilde A$.

To define design matrices with block entries we will need the following definition.
\begin{define}[well-spread set]\label{def-wellspread}
Let $S = \{A_1,\ldots,A_s\} \subset \M_{r,c}(1,1)$ be a  set (or multiset) of $s$ complex $r \times c$ matrices. We say that $S$ is {\em well-spread} if, for every subspace $V \subset \C^c$ we have
\begin{equation*}
	\sum_{i \in [s]}\dim(A_i(V)) \geq \frac{rs}{c}\cdot \dim(V). 
\end{equation*}
\end{define}

The following definition extends the definition of design matrices given in \cite{BDWY12}.

\begin{define}[design matrix]
	A matrix $A \in \M_{m,n}(r,c)$ is called a $(q,k,t)$-design matrix if it satisfies the 
	following three conditions
	\begin{enumerate}
		\item Each row of $A$ has at most $q$ non zero blocks.
		\item Each column of $A$ contains $k$ blocks that, together, form a well-spread set.
		\item For any $j \neq j' \in [n]$ there are most $t$ values of $i \in [m]$ so that both 
		$A_{ij}$ and $A_{ij'}$  are non-zero blocks. In other words, the supports of two 
		columns intersect in at most $t$ positions.
	\end{enumerate}
\end{define}

\begin{comm}\label{com-nonsingular}
	Notice that, for the case $r=c=1$, the second item simply requires that each column has at least $k$ non-zero entries. Hence, this definition extends the definitions of design matrices from previous works (\cite{BDWY12}, \cite{DSW12}). More generally, if $r=c$ then the second item is equivalent to asking that each column contains at least $k$ non singular blocks.
\end{comm} 

Our main theorem is the following lower bound on the rank of design matrices. Setting $r=c=1$ we  recover the rank bound from \cite{DSW12}. 

\begin{thm}[rank of design matrices]\label{thm-rank-design}
	Let $A \in \M_{m,n}(r,c)$ be a $(q,k,t)$-design matrix. Then $$ \rank(A) \geq cn - \frac{cn}{1+X},$$ with $X = \frac{kr}{ct(q-1)}.$
\end{thm}

We now describe the various geometric applications of this theorem.

\subsection{Projective rigidity}

Given a finite set of points $V$ in $\C^2$ containing some collinear triples, we can apply any 
projective transformation on $V$ and keep all collinear triples collinear. This gives $8$ `degrees of freedom'  
for us to `move' $V$ (keeping its collinearity structure). But are there more transformations we can 
perform? To study this question more formally, we begin with the following definition.

\begin{define}[Projective Rigidity]
Let $V = (v_1,\ldots,v_n) \in (\C^d)^n$ be a list of $n$ points in $\C^d$ and let $T \subset {[n] \choose 3}$ 
be a multiset of triples on the set $[n]$ (we allow repetitions of triples for technical reasons). Let 
$\cK_T \subset \C^{nd}$ be the variety of lists of $n$ points in which all triples in $T$ are collinear. 
Let $P_V \in \C^{nd}$ denote the concatenation of coordinate vectors of all points in $V$. We say that 
$(V,T)$ is $r$-rigid if $P_V$ is a non singular  point of $\cK_T$ and the dimension of its irreducible component 
is at most $r$. We denote the set of pairs $(V,T)$ as above (with $P_V \in \cK_T$) by $COL(n,d)$.
\end{define}

Hence, showing that a point set $V \subset \C^2$ with a family of triples $T$ is $8$-rigid means showing that it cannot be changed smoothly in any nontrivial way. Using our rank bound for design block matrices, we are able to prove a general theorem (Theorem~\ref{thm-rigidity-triples}) giving quantitative bounds on the rigidity of pairs $(V,T)$ satisfying certain conditions. For example, if every pair of points in $V$ is in exactly one triple in $T$ and no line contains more than half of the points in $V$ then we can  prove an upper bound of 15 on the rigidity of the pair $(V,T)$.  We refer the reader to Section~\ref{sec-rigidity} for a more complete description of these results.

\paragraph{Other notions of rigidity:} A more well-studied notion of geometric rigidity has to do with fixing the {\em distances} between pairs of points. Let $G = G(V, E)$ be a graph, where $|V| = n, \ |E| = m$. Let $\bp = (p_v)_{v \in V}$ be an 
embedding of $G$ in $\R^d$, where to each vertex $v \in V$ we assign the point 
$p_v \in \R^d$. By fixing the order of the vertices in $V$, we can identify the set of embeddings of $G$
in $\R^d$ with points $\bp \in (\R^d)^n = \R^{dn}$. Given such point-bar framework $(G, \bp)$, one is generally 
interested in the study of all continuous paths in $\R^{dn}$ which preserve 
the distances of all pairs of points in $E$. More succinctly, given the distance function of $G$
$\Delta_G : \R^{dn} \to \R^{m}$ defined by 
$$\Delta_G(x_1, \ldots, x_n) = (\| x_u - x_v \|_2)_{\{u, v\} \in E},$$
we are interested in studying all continuous paths in $\R^{dn}$ starting from $\bp$ 
which leave $\Delta_G$ unchanged.

If, for a given framework $(G, \bp)$, it turns out that every continuous path from $\bp$ which preserves
$\Delta_G$ terminates at a point $\bq \in \R^{dn}$ such that $\bq$ is an isometry of $\bp$, we say that
the framework $(G, \bp)$ is \emph{rigid}. That is, if 
$\Delta_G(\bp) = \Delta_G(\bq)$ implies $\Delta_{K_n}(\bp) = \Delta_{K_n}(\bq)$ for all $\bq \in \R^{dn}$
obtained from $\bp$ in the above manner, we say that the framework $(G, \bp)$ is rigid.
Otherwise, we call the framework $(G, \bp)$ \emph{flexible}. For more concrete motivations to the study
of rigidity, we refer the reader to~\cite{laman1970} and references therein.
%

A different notion of rigidity, which is closer to ours in spirit, is the one given by Raz~\cite{raz2016configurations},
which we now define. Given a (multi)set of lines $L = \{\ell_1, \ldots, \ell_n \}$ in $\C^3$, we define the 
\emph{intersection graph} of $L$ as the graph $G_L = G_L([n], E)$ where $\{i, j\} \in E$ iff $i \neq j$
and the corresponding lines $\ell_i$ and $\ell_j$ intersect. For a graph $G$, we say that $L$ is a realization
of $G$ if $G \subseteq G_L$. With these definitions, we say that a graph $G$ is rigid if for any \emph{generic}
realization $L = \{\ell_1, \ldots, \ell_n \}$ of $G$, we must have $G_L = K_n$. 

\subsection{Sylvester-Gallai for subspaces}

Another application of Theorem~\ref{thm-rank-design} gives a quantitative improvement to the results of \cite{DH} who generalized the Sylvester-Gallai theorem for arrangements of subspaces in $\C^d$.  We show the following:

\begin{thm} \label{thm-tightsg}
Let $V_1,V_2,\ldots,V_n\subset \C^d$ be $\ell$-dimensional subspaces such that $V_{i}\cap V_{i'}=\{\vec{0}\}$ for all $i\neq i'\in[n]$. Suppose that, for every $i_1\in[n]$ there exists at least $\delta (n-1)$ values of $i_2\in[n]\setminus\{i_1\}$ such that $V_{i_1}+V_{i_2}$ contains some $V_{i_3}$ with $i_3 \not\in \{i_1,i_2\}$. Then $$\dim(V_1+V_2+\cdots+V_n)\leq \left\lceil \frac{4\ell}{\delta} \right\rceil - 1.$$ 
\end{thm}

The original bound proven in \cite{DH} was a slightly worse $O(\ell^4/\delta^2)$.  For $\delta=1$ and $\ell=1$ the bound of $3$ we get is completely tight as there are three dimensional configurations of one dimensional subspaces  over $\C$ with every pair spanning some third subspace (this can be obtained by taking the Hesse configuration and moving to projective space \cite{HessePencil}). When $\delta=1$ and $\ell >1$ it remains open whether or not the bound $4\ell-1$ bound is tight or not  (one can get a lower bound of $3\ell$ by taking the product of the one dimensional example). 

The condition $V_i \cap V_{i'} = \{\vec{0}\}$ is needed due to the following example given in \cite{DH}: Set $\ell=2$ and $n=d(d-1)/2$ and let $\{\vec{e}_1,\vec{e}_2,\ldots,\vec{e}_d\}$ be the standard basis of $\C^d$. Define the $n$ spaces to be  $V_{ij} = \spn\{\vec{e}_i,\vec{e}_j\}$  with $1\leq i<j\leq d$. Now, for each $(i,j) \neq (i',j')$ the sum $V_{ij} + V_{i'j'}$ will contain a third space (since the size of $\{i,j,i',j'\}$ is at least three). However, this arrangement has dimension $d > \sqrt{n}$.

The one dimensional ($\ell=1$) version of Theorem~\ref{thm-tightsg} was originally proven in \cite{BDWY12,DSW12} as an application of the rank bound for (scalar) design matrices. In \cite{DH}, a different, more lossy, proof technique was developed to handle the  higher dimensional case (also relying on methods similar to \cite{BCCT}). Our proof goes back to the original proof strategy using rank of design matrices, now with block entries, and applying Theorem~\ref{thm-rank-design}.

\subsection{Pairwise incidences of lines and curves}

Our final application  of Theorem~\ref{thm-rank-design} is the following result about pairwise incidences in a given set of lines.

\begin{thm}\label{thm-lineincidence-intro}
	Let $L_1,\ldots,L_n \subset \C^d$ be lines such that each $L_i$ intersects at least $k$ other lines and, among those $k$ lines, at most $k/2$ have the same intersection point on $L_i$. Then, the $n$ lines are contained in an affine subspace of dimension at most $\left\lfloor \frac{4n}{k+2}\right\rfloor -1$.
\end{thm}

We also prove an analog of Theroem~\ref{thm-lineincidence-intro} for higher degree curves. We refer to a curve as a degree $r$ parametric curve if it is given as the image of a polynomial map (in one variable) of degree at most $r$.

\begin{thm}\label{thm-curveincidence-intro}
	 Let $\gamma_1,\ldots,\gamma_n \subset \C^d$ be degree $r$ parametric curves such that each $\gamma_i$  has at least $k$  incidences with the other curves such that, among those $k$ incidences, at most $k/2r$ have the same intersection point on $\gamma_i$. Then, the $n$ curves are contained in a subspace of dimension at most $ \frac{2(r+1)^4n}{k}$.
\end{thm}

\subsection{Organization}
In Section~\ref{sec-scalingcapacity} we develop the necessary machinery for scaling of matrices with block entries. In Section~\ref{sec-rankdesign} we prove our main theorem, Theorem~\ref{thm-rank-design}. In Section~\ref{sec-rigidity} we give the applications for geometric rigidity. In Section~\ref{sec-highsg} we prove our improved Sylvester-Gallai theorem for subspaces (Theorem~\ref{thm-tightsg}). In Section~\ref{sec-incidence} we prove Theorem~\ref{thm-lineincidence-intro} and its generalization for higher degree curves.

\section{Matrix scaling and capacity}\label{sec-scalingcapacity}
In this section we develop the machinery needed to prove Theorem~\ref{thm-rank-design}.
We denote by $I_s \in \M_{s,s}(1,1)$ the $s \times s$ identity matrix. For a matrix $A$ 
we denote $\|A\|_2^2 = \tr(AA^*).$

\begin{define}[Row Normalization]
	Let $A \in \M_{m,n}(r,c)$. For each $i \in [m]$ let $$R_i(A) = \sum_{j=1}^n A_{i, j}A_{i, j}^* \in \M_{r,r}(1,1).$$ If all matrices $R_i(A)$ are non singular (and hence, positive definite) we define the {\em row normalizing matrix} of $A$ as the matrix $R(A) \in \M_{m,m}(r,r)$ whose diagonal blocks are the matrices $R(A)_{i,i} = (R_i(A))^{-1/2}.$ We define the {\em row normalization} of $A$ as the product
   $\row(A) = R(A)\cdot A$. If $R_i(A) = I_r$ for all $i \in [m]$ we say that $A$ is {\em row normalized}.
\end{define}

\begin{define}[Column Normalization]
Let $A \in \M_{m,n}(r,c)$. For each $j \in [n]$ let $$C_j(A) = \frac{nc}{mr}\sum_{i=1}^m A_{i, j}^*A_{i, j} \in \M_{c,c}(1,1).$$ If all matrices $C_j(A)$ are non singular (and hence, positive definite) we define the {\em column normalizing matrix} of $A$ as the matrix $C(A) \in \M_{n,n}(c,c)$ whose diagonal blocks are the matrices $C(A)_{j,j} = \left(C_j(A)\right)^{-1/2}.$ We define the {\em column normalization} of $A$ as the product
   $\col(A) = A\cdot C(A)$. If $C_j(A) = I_c$ for all $j \in [n]$ we say that $A$ is {\em column normalized}.
\end{define}

\begin{define}[Doubly stochastic block matrices] A matrix $A \in \cM_{m,n}(r, c)$ is said to be
\emph{doubly stochastic} if it is both row normalized and column normalized.	We define the distance of $A$ to a doubly stochastic matrix, denoted by $\ds(A)$, 
	as 
	$$ \ds(A) = \sum_{j=1}^n \left\|C_j(A) - I_c\right\|_2^2 + \sum_{i=1}^m \left\| R_i(A) - I_r \right\|_2^2.$$
\end{define}

\begin{define}[Matrix scaling]
Let $A \in \M_{m,n}(r,c)$. A {\em scaling} of $A$ is a matrix $B \in \M_{m,n}(r,c)$ obtained as follows: Let $R_1,\ldots,R_m \in \M_{r,r}(1,1)$ and $C_1,\ldots,C_n \in \M_{c,c}(1,1)$ be non-singular complex matrices. We refer to the $R_i$'s as {\em row scaling coefficients} and to the $C_j$'s as {\em column scaling coefficients}. Now, we let $B_{ij} = R_i\cdot A_{ij}\cdot C_j$. Notice that, if $B$ is a scaling of $A$ than $A$ is a scaling of $B$. 
\end{define}

We would like to understand when a matrix has a doubly stochastic scaling. For technical reasons, it is more natural to ask when a matrix can be scaled to be arbitrarily close to doubly stochastic. This question turns out to have a much nicer answer and, for  our purposes, an `almost' doubly stochastic matrix will do just fine.

\begin{define}[Scalable Matrices]
A matrix $A \in \M_{m,n}(r,c)$ is said to be {\em scalable} if, for every $\eps >0$ there exist a scaling $B$ of $A$ such that $\ds(B) \leq \eps$. 
\end{define}

Our goal is to give sufficient conditions for a matrix to be scalable. For this we need to define a measure called {\em capacity} which is a generalization of capacity of non-negative matrices defined in \cite{GurYianilos} (used to study Sinkhorn's algorithm) and a special case of capacity of operators defined in \cite{gurvits2004} (used to study an operator generalization of Sinkhorn's algorithm).

\begin{define}[Capacity]\label{def:capacity}
	The capacity of a block matrix 	$A \in \M_{m,n}(r, c)$ is defined as:	
	$$\capac(A) =  \inf \left\{ \prod_{j=1}^n \det\left(  \frac{nc}{mr}\sum_{i=1}^m A_{ij}^* X_i A_{ij} \right) \ : \
	X_i \succ 0  \text{ and } \prod_{i=1}^m \det(X_i) = 1 \right\}.$$ Where  the  $X_i$'s  are $r\times r$ complex Hermitian positive definite matrices.
\end{define}

The main technical result of this section is given in the following theorem. We will prove it at the end of the section, following some preliminaries. The proof will mimic  the analog result for scalar matrices (Sinkhorn's algorithm) using alternate left/right scaling and using the capacity as a progress measure (this is also the approach taken in \cite{GGOW15} for operator scaling).
\begin{lem}\label{lem-capacity-scalable}
	Let $A \in \M_{m,n}(r,c)$. If $\capac(A)>0$ then $A$ is scalable.
\end{lem}

The proof of the lemma will be using an iterative algorithm that, at each step performs row/column normalization of $A$. We will show that this process must converge to a doubly stochastic matrix. We start with some useful claims. The first claim relates the capacity of $A$ with that of $A^*$. For our purposes, we will only need to use the fact that, if one of them is zero, then so is the other.

\begin{claim}\label{cla-transpose}
Let $A \in \M_{m,n}(r,c)$. Then
$$
\capac(A)^{1/nc} = \frac{nc}{mr} \cdot \capac(A^*)^{1/mr}
$$
\end{claim}

\begin{proof}
Let $\text{PD}_k$ denote the set of $k \times k$ Hermitian positive definite matrices. Notice that
\begin{align*}
\capac(A)^{1/nc} &= \inf \left\{ \prod_{j=1}^n \det\left(  \frac{nc}{mr}\sum_{i=1}^m A_{ij}^* X_i A_{ij} \right)^{1/nc} \ : \
	X_i \succ 0  \text{ and } \prod_{i=1}^m \det(X_i) = 1 \right\} \\
&=  \inf \left\{ \frac{\prod_{j=1}^n \det\left(  \frac{nc}{mr}\sum_{i=1}^m A_{ij}^* X_i A_{ij} \right)^{1/nc}}{ \prod_{i=1}^m \det(X_i)^{1/mr}} \ : \
	X_i \succ 0  \right\}.
\end{align*}
Similarly
$$
\capac(A^*)^{1/mr} = \inf \left\{ \frac{\prod_{i=1}^m \det\left(  \frac{mr}{nc}\sum_{j=1}^n A_{ij} Y_j A_{ij}^* \right)^{1/mr}}{ \prod_{j=1}^n \det(Y_j)^{1/nc}} \ : \
	Y_j \succ 0  \right\}.
$$
Suppose for now that, $\capac(A^*)$ is non-zero. We have
\begin{align*}
&\frac{\capac(A)^{1/nc}}{\capac(A^*)^{1/mr}} = \\ &\text{inf}_{X_i \in \text{PD}_r} \: \text{sup}_{Y_j \in \text{PD}_c} \: \left\{ \frac{\prod_{j=1}^n\left[ \det\left(  \frac{nc}{mr}\sum_{i=1}^m A_{ij}^* X_i A_{ij} \right) \cdot  \det(Y_j) \right]^{1/nc}}{\prod_{i=1}^m \left[ \det\left(  \frac{mr}{nc}\sum_{j=1}^n A_{ij} Y_j A_{ij}^* \right) \cdot  \det(X_i) \right]^{1/mr}} \right\} \ge \\
& \text{inf}_{ X_i \in \text{PD}_r} \: \left\{ \frac{\prod_{j=1}^n\left[ \det\left(  \frac{nc}{mr}I_c \right) \right]^{1/nc}}{\prod_{i=1}^m \left[ \det\left(  \frac{mr}{nc}\sum_{j=1}^n A_{ij} \widetilde{Y_j} A_{ij}^* \right) \cdot  \det(X_i) \right]^{1/mr}} \right\} =\\
&\frac{nc}{mr} \cdot  \text{inf}_{X_i \in \text{PD}_r} \: \left\{ \frac{1}{\prod_{i=1}^m \left[ \det\left(  \frac{mr}{nc}\sum_{j=1}^n A_{ij} \widetilde{Y_j} A_{ij}^* \right) \cdot  \det(X_i) \right]^{1/mr}} \right\},
\end{align*}
where $\widetilde{Y_j} = \left( \sum_{i=1}^m A_{ij}^* X_i A_{ij}\right)^{-1}$. Continuing:
\begin{align*}
\frac{\capac(A)^{1/nc}}{\capac(A^*)^{1/mr}} &\ge \frac{nc}{mr} \cdot  \text{inf}_{ X_i \in \text{PD}_r} \: \left\{ \frac{1}{\frac{1}{nc} \cdot \sum_{i=1}^m   \sum_{j=1}^n \tr \left[A_{ij} \widetilde{Y_j} A_{ij}^*X_i \right]} \right\} \\
&=\frac{nc}{mr} \cdot  \text{inf}_{X_i \in \text{PD}_r} \: \left\{ \frac{1}{\frac{1}{nc} \cdot \sum_{j=1}^n   \tr\left[\sum_{i=1}^m \widetilde{Y_j} A_{ij}^*X_i A_{i,j}\right]} \right\} \\
&= \frac{nc}{mr} \cdot  \text{inf}_{X_i \in \text{PD}_r} \: \left\{ \frac{1}{\frac{1}{nc} \cdot \sum_{j=1}^n   \tr\left[I_c\right]} \right\} \\
&= \frac{nc}{mr},
\end{align*}
where the first inequality follows from the AM-GM inequality, applied to the (non negative) eigenvalues of a PSD matrix. In the other direction, we apply a similar argument  to $A^*$ to obtain  that, if $\capac(A)$ is nonzero then
$$
\frac{\capac(A^*)^{1/mr}}{\capac(A)^{1/nc}} \ge \frac{mr}{nc}
$$
Rearranging completes the proof. 
\end{proof}

\begin{claim}[Capacity of normalized matrices]\label{cla-capac-normalized}
	Let $A \in \M_{m,n}(r,c)$ be a column-normalized matrix. Then $\capac(A) \leq 1$.
\end{claim}

\begin{proof}
	Notice that 
		\begin{eqnarray*}
			\capac(A) &\leq& \prod_{j=1}^n \det\left( \frac{nc}{mr}\sum_{i=1}^m A_{ij}^* I_r A_{ij} \right) = \prod_{j=1}^n \det(I_c) = 1.
		\end{eqnarray*}
\end{proof}

\begin{claim}[Capacity of a scaling]\label{cla-capofscaling}
		Let $A \in \M_{m,n}(r,c)$ and let $B$ be a scaling of $A$ with row scaling coefficients $R_1,\ldots,R_m \in \M_{r,r}(1,1)$ and column scaling coefficients $C_1,\ldots,C_n \in \M_{c,c}(1,1)$. Then
				$$ \capac(B) =  \left(\prod_{j=1}^n |\det(C_j)| \right)^2 \left(\prod_{i=1}^m |\det(R_i)| \right)^{2nc/mr} \cdot\capac(A). $$
\end{claim}

\begin{proof}
	\begin{align*}
		\capac(B) 
		&= \inf \left\{ \prod_{j=1}^n \det\left( \frac{nc}{mr}\sum_{i=1}^m 
		C_j^*  A_{ij}^* R_i^* X_i R_i A_{ij} C_j  \right) \ : \
		X_i \succ 0 \text{ and } \prod_{i=1}^m \det(X_i) = 1  \right
		\} \\
		&= \left(\prod_{j=1}^n |\det(C_j)| \right)^2 \cdot \inf \left\{  \prod_{j=1}^n \det\left( \frac{nc}{mr}\sum_{i=1}^m  A_{ij}^* Y_i  A_{ij} \right) \ : \
		Y_i \succ 0 \text{ and } \prod_{i=1}^m \det(Y_i) = \prod_{i=1}^m |\det(R_i)|^2  \right\} \\
		&=   \left(\prod_{j=1}^n |\det(C_j)| \right)^2\left(\prod_{i=1}^m |\det(R_i)| \right)^{2nc/mr} \cdot \capac(A),
	\end{align*}
Where the last equality is obtained by observing the effect of scaling all the $Y_i$'s by the same constant $\alpha$ on the capacity.
\end{proof}

To prove a quantitative bound on the rate of growth of the capacity under row/column scaling we will need the following quantitative variant of the AM-GM inequality. A proof (along the lines of \cite{LSW}) is given  for completeness.

\begin{claim}(Quantitative AM-GM)\label{cla-quant-AMGM}
Let $x_1,\ldots,x_s>0$ be real numbers such that $\sum_{i \in [s]}x_i=s$ and $\sum_{i \in [s]} (x_i -1)^2 = \eps$. Then
$$ \prod_{i\in [s]}x_i \leq \max\left\{ e^{-\eps/6}, e^{-1/6} \right\}.$$
\end{claim}
\begin{proof}
First assume that $\eps \leq 1$. Let $y_i = x_i-1$ so that $\sum_i y_i = 0$ and $\sum_i y_i^2 = \eps$. Using the inequality $1+t \leq e^{t-t^2/2+t^3/3}$ which holds for all real  $t$ we get that
\begin{eqnarray*}
	\prod_{i=1}^s(1+y_i) &\leq & \exp\left(-\frac{1}{2}\sum_i y_i^2 + \frac{1}{3}\sum_i y_i^3 \right) \\
	&\leq & \exp\left( -\frac{1}{2}\sum_i y_i^2 + \frac{1}{3}\left(\sum_i y_i^2\right)^{3/2} \right) \\
	&\leq & \exp( \eps/6 ),
\end{eqnarray*}
where the last inequality used the fact that $\eps\leq 1$.  To argue about values of $\eps$ larger than $1$ we observe that the function $f(z) = \prod_i (1 + zy_i)$ is decreasing in the range $0 \leq z \leq 1$. To see this, notice that the derivative of $\ln f(z)$ is precisely $\sum_i \frac{y_i}{1+zy_i} \leq \sum_i y_i = 0$. Since $\ln f(z)$ is decreasing, $f(z)$ is also decreasing. Hence, we can apply the bound for small $\eps$ to get that $f(1) \leq f(z^*)\leq \exp(-1/6)$ for $z^* = \eps^{-1/2}\leq 1$.
\end{proof}

\begin{claim}[Capacity and row/column normalization]\label{cla-capacity-progress}
	Let $A \in \M_{m,n}(r,c)$ be a matrix such that $\ds(A) = \eps$. Then,
	\begin{enumerate}
		\item If $A$ is column-normalized, then $\capac(\row(A)) \geq \capac(A)$ (assuming $\row(A)$ is  defined).
		\item If $A$ is row-normalized then $\capac(\col(A)) \geq \min\left\{ e^{1/6},e^{\eps/6} \right\} \cdot \capac(A)$ (assuming $\col(A)$ is  defined).
	\end{enumerate}
		
\end{claim}

\begin{comm}
One can prove a similar quantitative bound in terms if $\eps$ also in item (1) but we will not need it.
\end{comm}

\begin{proof}
We start by proving the first item. Let $R_i(A) = \sum_{j=1}^n A_{ij}A_{ij}^*$. Since the scaling coefficients used to get $\row(A)$ from $A$ are $R_i(A)^{-1/2}$, we get that, by Claim~\ref{cla-capofscaling}, 
\begin{equation}\label{eq-caprowA}
	\capac(\row(A)) = \left( \prod_{i=1}^m \det(R_i(A))^{-1/2} \right)^{2nc/mr}\cdot \capac(A).
\end{equation}
Let $\lambda_{i1},\ldots,\lambda_{ir}>0$ denote the  eigenvalues of the (positive definite) matrix $R_i(A)$. Since $A$ is column normalized we have that
\begin{eqnarray*}
	\sum_{i=1}^m \sum_{k=1}^r \lambda_{ik} &=& \sum_{i=1}^m \tr\left( \sum_{j=1}^n A_{ij}A_{ij}^*\right) \\
	&=& \sum_{j=1}^n \tr \left( \sum_{i=1}^m A_{ij}^*A_{ij} \right) \\
	&=& \sum_{j=1}^n \tr\left( \frac{rm}{nc} \cdot I_c \right) = rm.
\end{eqnarray*}
Hence, by the AM-GM inequality we get that
$$ \prod_{i=1}^m \det(R_i(A)) = \prod_{i=1}^m \prod_{k=1}^r \lambda_{ik} \leq 1. $$
Plugging this into Eq.~\ref{eq-caprowA} proves the first part of the claim. 

To prove the second part, let $C_j(A) = \frac{nc}{mr}\sum_{i=1}^m A_{ij}^* A_{ij}$ and recall  that $\col(A)$ is obtained from $A$ by scaling the columns with coefficients $C_j(A)^{-1/2}$. Hence, by Claim~\ref{cla-capofscaling}, we have
\begin{equation}\label{eq-capcolA}
 \capac(\col(A)) = \left( \prod_{j=1}^n \det(C_j(A))^{-1/2} \right)^2 \cdot \capac(A).
\end{equation}
Let $\mu_{j1},\ldots,\mu_{jc}$ be the  eigenvalues of $C_j(A)$. As before, we have that
$$ \sum_{j=1}^n\sum_{k=1}^c \mu_{jk} = \sum_{j=1}^n \tr\left(\frac{nc}{mr}\sum_{i=1}^m A_{ij}^*A_{ij} \right) = \frac{nc}{mr} \sum_{i=1}^m \tr(I_r) = nc.$$
Using the assumption $\ds(A)=\eps$ and the fact that $A$ is row normalized we can also deduce that
\begin{equation*}
	\sum_{j=1}^n \sum_{k=1}^c (\mu_{jk}-1)^2 = \sum_{j=1}^n \tr\left( (C_j(A) - I_c)^2 \right) = \eps.
\end{equation*}
Hence, we can use Claim~\ref{cla-quant-AMGM} to obtain the bound
$$ \prod_{j=1}^n \det(C_j(A)) = \prod_{j=1}^n\prod_{k=1}^c \mu_{jk} \leq \max\left\{ e^{-\eps/6}, e^{-1/6} \right\}. $$
Plugging this into Eq.~\ref{eq-capcolA} proves the second part of the claim.
\end{proof}

Another useful claim:

\begin{claim}\label{cla-rowcoldefined}
Let $A \in \M_{m,n}(r,c)$ be such that $\row(A)$ and $\col(A)$ are well defined. Then,  $\row(B)$ and $\col(B)$ are well defined for every scaling $B$ of $A$.
\end{claim}
\begin{proof}
	Suppose $B_{ij} = R_iA_{ij}C_j$ for non-singular scaling coefficients $R_1,\ldots,R_m \in \M_{r,r}(1,1)$ and $C_1,\ldots,C_n \in \M_{c,c}(1,1)$. To  show that $\row(B)$ is well defined we need to argue that, for each $i \in [m]$, the PSD matrix 
	$$R_i(B) = \sum_{j=1}^n B_{ij}B_{ij}^* = \sum_{j=1}^n R_iA_{ij}C_jC_j^* A_{ij}^*R_i^*$$
	is non singular. We can take out the non singular $R_i$ and $R_i^*$ factors and so we need to show that 
	$$ \sum_{j=1}^n A_{ij}C_jC_j^* A_{ij}^*$$ 
	is non singular. This $r\times r$ PSD matrix is singular iff there exists a vector $v \in \C^r$ so that $C_j^*A_{ij}^* v = 0$ for all $j\in [n]$. Since the $C_j$'s are non singular, such a $v$  would also be in the kernel of $R_i(A) = \sum_{j=1}^n A_{ij}A_{ij}^*$ in contradiction to our assumption that $\row(A)$ is well defined. The proof for $\col(A)$ is identical.
\end{proof}

\begin{proof}[Proof of Lemma~\ref{lem-capacity-scalable}]
Let $A_0=\col(A)$ and define recursively $$ A_{k+1} = \col(\row(A_k)).$$ 
Notice that $\col(A)$ is well defined since $\capac(A)>0$ and that $\row(A)$ is well defined since $\capac(A^*) > 0$ (using Claim~\ref{cla-transpose}). Hence, by Claim~\ref{cla-rowcoldefined},  this property will remain true for all matrices $A_k$ in the sequence (since they are all scalings of $A$). We wish to show that $\ds(A_k)$ approaches zero when $k$ goes to infinity. Assume in contradiction that $\ds(A_k) \geq \eps$ for some  $0<\eps<1$ and all $k \geq 0$. Applying Claim~\ref{cla-capacity-progress}, we get that $$ \capac(A_{k+1}) \geq \exp(\eps/6)\cdot \capac(A_{k}).$$ The matrices $A_k$ are all column-normalized and so, by Claim~\ref{cla-capac-normalized}, $\capac(A_k) \leq 1$ for all $k \geq 0$. This gives a contradiction  to the claimed growth of $\capac(A_k)$. 
\end{proof}

\subsection{Bounding the capacity of a matrix}

In this section we will develop  machinery useful for  proving that the capacity of certain matrices is positive.

\begin{claim}\label{cla-capacity-monotone}
Let $A,B \in \M_{m,n}(r,c)$ be two block matrices such that, for every $i\in [m],j\in [n]$, $B_{ij}$ is either equal to $A_{ij}$ or equal to a zero $r \times c$ block. Then, $\capac(A) \geq \capac(B)$. In particular, if $\capac(B)>0$ then $\capac(A)>0$.	
\end{claim}
\begin{proof}
	The claim following from the simple fact that, for two PSD matrices $X,Y$, we have $\det(X+Y) \geq \det(X)$. Using this in the definition of capacity, we see that, replacing some blocks in $A$ with zeros can only decrease the product of determinants being minimized.
\end{proof}

\begin{claim}[Block diagonal matrices]\label{cla-block-diagonal}
Suppose $M$ is an $s \times s$ block diagonal matrix with entries $M_{ij} \in \M_{m,n}(r,c)$. Then, viewing $M$ as an element of $\M_{sm,sn}(r,c)$ we have $\capac(M) = \prod_{i=1}^s\capac(M_{ii})$. In particular, if all the $M_{ii}$'s have positive capacity, then so does $M$.
\end{claim}
\begin{proof}
	To save on notations, we will only prove the claim for $s=2$ (the general case is proved along the same lines). Suppose therefore that $M$ has diagonal blocks $A,B \in \M_{m,n}(r,c)$ and zero blocks in the two off diagonal positions. More precisely, viewing $M$ as an element of $\M_{2m,2n}(r,c)$ (and treating $M_{ij}$ as the actual $r \times c$ blocks of $M$),  we have $M_{ij} = A_{ij}$ for $1 \leq i \leq m$ and $1 \leq j \leq n$, $M_{ij} = B_{(i-m)(j-n)}$ for $m+1 \leq i \leq 2m, n+1 \leq j \leq 2n$ and $M_{ij}=0$ for all other pairs $i,j$. 
	
	To see that the capacity splits into the product of capacities, it is enough to rewrite the capacity in a scale invariant form:
	\begin{eqnarray*}
		\capac(M) &=&  \inf \left\{ \prod_{j=1}^{2n} \det\left(  \frac{2nc}{2mr}\sum_{i=1}^{2m} M_{ij}^* X_i M_{ij} \right) \ : \
	X_i \succ 0  \text{ and } \prod_{i=1}^{2m} \det(X_i) = 1 \right\} \\
	&=& \inf \left\{ \frac{\prod_{j=1}^{2n} \det\left(  \frac{2nc}{2mr}\sum_{i=1}^{2m} M_{ij}^* X_i M_{ij} \right)}{\left( \prod_{i=1}^{2m} \det(X_i) \right)^{2nc/2rm}} \ : \
	X_i \succ 0   \right\} \\
	&=& \inf \left\{ \frac{\prod_{j=1}^{n} \det\left(  \frac{nc}{mr}\sum_{i=1}^{m} A_{ij}^* X_i A_{ij} \right)}{\left( \prod_{i=1}^{m} \det(X_i) \right)^{nc/rm}}\cdot \frac{\prod_{j=1}^{n} \det\left(  \frac{nc}{mr}\sum_{i=1}^{m} B_{ij}^* Y_i B_{ij} \right)}{\left( \prod_{i=1}^{m} \det(Y_i) \right)^{nc/rm}} \ : \
	X_i,Y_i \succ 0   \right\} \\
	&=& \capac(A)\cdot \capac(B).
	\end{eqnarray*}
\end{proof}

\subsubsection{A result from Brascamp-Lieb theory}

We will rely on a technical result from \cite{BCCT} (Proposition 5.2 in that paper)  that allows to bound the capacity of a block matrix with only one column (i.e., a set of matrices). The results of \cite{BCCT} are stated for real PSD matrices but the proofs carry over easily to the complex Hermitian case. To make the connection to \cite{BCCT} easier to see we first give some definitions from \cite{BCCT}. The first notion is that of a {\em Brascamp-Lieb datum} $\bf (B,p)$ with ${\bf B} = (B_1,\ldots,B_k)$ a set of linear transformations $B_j : H \mapsto H_j$ between Hilbert spaces  and ${\bf p} = (p_1,\ldots,p_k)$ a sequence of positive real numbers. For our purposes it is enough to treat the case when $H= \C^c$ and, for all $1 \leq j \leq k$ we have  $H_j = \C^r$ and $p_j = c/kr$. In the notations of \cite{BCCT} this datum satisfies condition `(7)' which requires that $\dim(H) = \sum_j p_j \dim(H_j)$. Given such a datum (we ignore the vector $\bf p$ since it is fixed), a quantity called ${\rm BL_g}({\bf B})$ is defined (the subscript $\rm g$ stands for `Gaussian').
\begin{equation*}
	{BL_g}({\bf B}) := \sup\left\{ \left(\frac{\prod_j \det(X_j)}{(c/kr)\det(\sum_j B_j^* X_j B_j)}\right)^{1/2} \, : \, X_j \succ 0 \right\}.
\end{equation*}
Using scale invariance (scaling each $X_j$ by the same constant does not change the ratio) this is the same as
\begin{equation*}
	{BL_g}({\bf B}) := \sup\left\{ \left( (c/kr)\det(\sum_j B_j^* X_j B_j)\right)^{-1/2} \, : \, X_j \succ 0 , \prod_j\det(X_j) = 1 \right\}
\end{equation*}
Going back to our notations, if $A \in \M_{k,1}(r,c)$ is a block matrix with one column comprised of blocks $A_{11},\ldots,A_{k1} \in \M_{r,c}(1,1)$ then, the capacity of $A$ is positive iff the quantity ${\rm BL_g}$ is bounded for the datum composed of the blocks of $A$ (treated as maps from $\C^c$ to $\C^r$). The following is a restatement of Proposition 5.2 from \cite{BCCT} (we do not require the `furthermore' part of the theorem).

\begin{thm}[\cite{BCCT}]\label{prop-BCCT}
Let ${\bf B} = (B_1,\ldots,B_k)$ be a Brascap-Lieb datum as above, which satisfies:
\begin{enumerate}
	\item Each $B_j$ is surjective and the common kernel of all $B_j$'s is trivial (`non degenerate datum' in the language of \cite{BCCT}).
	\item For each subspace $V$ of $\C^c$ we have $\dim(V) \leq (c/kr)\sum_j \dim(B_j(V))$ (condition `(8)' in \cite{BCCT}).
\end{enumerate}
Then, the quantity ${\rm BL_g({\bf B})}$ is bounded from above.
\end{thm}

Restated in our language this becomes:
\begin{thm}[\cite{BCCT}]\label{thm-BCCT}
	Let $A \in \M_{k,1}(r,c)$ be such that the blocks $A_{11},\ldots,A_{k1} \in \M_{r,c}(1,1)$ form a well-spread set (see Definition~\ref{def-wellspread}). Then $\capac(A) >0$. 
\end{thm}
\proof{
The first condition of Theorem~\ref{prop-BCCT} holds in our case using the fact that the blocks $A_{i1}$ are well spread. To see that each $A_{i1}$ is onto $\C^r$ apply the well-spread condition with $V = \C^c$. To see that their common kernel is trivial, apply the same bound with $V$ equal to their common kernel. The second condition in Theorem~\ref{prop-BCCT} is equivalent to our well-spread definition and so requires no proof. By the preceding discussion, the bound on $\rm BL_g$ implies that the capacity is positive.
}


\section{Rank of design matrices with block entries}\label{sec-rankdesign}
In this section we will prove Theorem~\ref{thm-rank-design}. First, we analyze a transformation taking any design matrix to another design matrix which is scalable.

\subsection{Regularization of a design matrix}

\begin{define}[Design matrix in regular form]
A $(q,k,t)$-design matrix $A \in \M_{m,n}(r,c)$ is in {\em regular form} if $m = nk$ and, in each column $i \in [n]$, the $k$ blocks $A_{(i-1)k+1,i}, \ldots, A_{(i-1)k+k,i}$ form a well-spread set. That is, the second item in the definition of a design matrix is satisfies by $k$-tuples of blocks that are row-disjoint in $A$.
\end{define}

\begin{claim}\label{cla-regularize}
Let $A \in \M_{m,n}(r,c)$ be a $(q,k,t)$-design matrix. Then, there exists a $(q,k,tq)$-design matrix $B \in \M_{nk,n}(r,c)$ in regular form such that  $\rank(B) \leq \rank(A)$.
\end{claim}
\begin{proof}
We construct $B$ in $n$ steps. In the first step we add to $B$ $k$ rows of $A$ so that their first column entries are well-spread. In the next step we add $k$ more rows to $B$ using the $k$ rows in $A$ in which the second column entries form a well spread set. We continue in this manner until we end up with $B$ having $nk$ rows. Since each row of $A$ contains at most $q$ non zero blocks, we have that each row of $A$ is repeated at most $q$ times in $B$. Hence, the supports of two columns in $B$ can intersect in at most $tq$ positions. Since all rows of $B$ are from $A$ the rank of $B$ cannot increase (it might decrease if we do not use all rows of $A$).	
\end{proof}

\begin{claim}\label{cla-regular-scalable}
Suppose $B \in \M_{nk,n}(r,c)$ is a $(q,k,t)$-design matrix in regular form. Then $B$ is scalable.
\end{claim}
\begin{proof}
%
We call the entries of $B$ in positions $((i-1)k+\ell,i)$ for $\ell \in [k]$ {\em special}. Let $B' \in \M_{nk,n}(r,c)$ be the matrix obtained from $B$ by replacing all the non special entries of $B$ by zero blocks. By Claim~\ref{cla-capacity-monotone} and Lemma~\ref{lem-capacity-scalable} it is enough to prove that $\capac(B')>0$. We can consider $B'$ as a diagonal $n \times n$ matrix with entries in $\M_{k,1}(r,c)$ and so, using Claim~\ref{cla-block-diagonal}, it is enough to show that the special entries in each column form a $\M_{k,1}(r,c)$ matrix with positive capacity. This follows from  Theorem~\ref{thm-BCCT} and using the assumption that the special entries in each column form a well spread set.
\end{proof}

\subsection{Proof of Theorem~\ref{thm-rank-design}}

We will use the following folklore lemma on diagonal dominant matrices.

\begin{lem}[Diagonal dominant matrices]\label{lem-diagonaldom}
	Let $H \in \M_{n,n}(1,1)$ be a square Hermitian complex matrix. Suppose $H_{i,i} \geq L > 0$ for all $i \in [n]$ and let $S = \sum_{i \neq j} |H_{i,j}|^2$. Then $$ \rank(H) \geq \frac{L^2n^2}{nL^2 + S}= n - \frac{nS}{nL^2 + S}.$$ We call a matrix $H$ satisfying these two conditions an $(L,S)$-{\em diagonal dominant} matrix.
\end{lem}
\begin{proof}
First, notice that we can assume w.l.o.g that $H_{i,i} = L$ for all $i$. Indeed, otherwise we  scale the $i$'th row and column by $0< \sqrt{L/H_{ii}} \leq 1$ to get a new Hermitian matrix with $L$ on the diagonal and with smaller $S$. Then,
	$$n^2L^2 = \tr(H)^2 \leq \rank(H) \tr(H^2) = \rank(H) \cdot \sum_{i,j}|H_{i,j}|^2 = \rank(H)\cdot (nL^2 + S).$$
\end{proof}

The following claim is an easy consequence of Cauchy-Schwartz (applied coordinate-wise)
\begin{claim}\label{cla-CS-block}
Let $A_1,\ldots,A_t \in \M_{r,c}(1,1)$ then 
$$ \left\| \sum_{i \in [t]}A_i \right\|_2^2 \leq t \cdot \sum_{i \in [t]}\|A_i\|_2^2.$$	
\end{claim}

Another useful claim:
\begin{claim}\label{cla-rowest}
Suppose $C_1,\ldots,C_q \in \M_{r,c}(1,1)$ are such that $\sum_{i\in[q]} C_iC_i^* = I_r$. Then
\begin{equation*}
	\sum_{i\neq j} \left\| C_i^*C_j \right\|_2^2 \leq r(1-1/q).
\end{equation*}	
\end{claim}
\begin{proof}
The sum in the claim is equal to the difference of the two sums:
$$S_1 - S_2 = \sum_{i,j}\left\|C_i^*C_j\right\|_2^2 - \sum_{i\in[q]}	\left\|C_i^*C_i\right\|_2^2. $$
First notice that
$$ S_1 = \sum_{i,j}\tr(C_i^*C_jC_j^*C_i)= \sum_{i,j}\tr(C_iC_i^*C_jC_j^*) = \tr(I_r^2) = r.$$ Next notice that, by Claim~\ref{cla-CS-block}, we have
$$S_2 = \sum_{i\in[q]}	\left\|C_i^*C_i\right\|_2^2 = \sum_{i\in[q]}	\left\|C_iC_i^*\right\|_2^2 \geq (1/q)\left\| I_r\right\|_2^2 = r/q.$$ These two calculations complete the proof.
\end{proof}

The bulk of the proof is given in the next lemma.
\begin{lem}\label{lem-design+scale=rank}
Suppose $M \in \M_{m,n}(r,c)$ is a $(q,k,t)$-design matrix that is scalable. Then $$ \rank(M) \geq nc - \frac{nc}{1+X},$$ with $$ X = \frac{mrq}{cnt(q-1)}.$$
\end{lem}
\begin{proof}
Since scaling does not change rank and preserves the property of being a $(q,k,t)$-design, we may assume w.l.o.g that $M$ is already scaled (for some $\eps$ that we will later send to zero). Notice that we could, w.l.o.g, assume that the `row sums' of $M$ are perfectly scaled and that the `error' is only in the column sums (just apply one additional row normalization). That is, 
\begin{enumerate}
	\item For all $i \in [m]$, $\sum_{j \in [n]} M_{ij}M_{ij}^* = I_r$.
	\item For all $j \in [n]$, $\sum_{i \in [m]} M_{ij}^*M_{ij} = \frac{mr}{nc} I_c + E(\eps)$, where $E(\eps)$ is a matrix that goes to zero (entry wise)  with $\eps$ going to zero. 
\end{enumerate}

Let $H = M^* M$ be $nc \times nc$ complex Hermitian matrix. We will show that $H$ is $(L,S)$-diagonal dominant with 
\begin{equation}\label{eq-L}
L = \frac{rm}{cn}  + o(1), \eps \mapsto 0
\end{equation}
and
\begin{equation}\label{eq-S}
	S \leq mtr(1 - 1/q) + o(1), \eps \mapsto 0.
\end{equation}
Equation (\ref{eq-L}) follows from the scaling condition on the columns of $M$ since the diagonal $c \times c$ blocks of $H$ are $\frac{rm}{cn}I_c$ plus error that vanishes with epsilon. We now turn to prove the bound (\ref{eq-S})
 on $S$ (the sum of squares of off-diagonal entries). We have
$$ 	S = \sum_{j \neq j' \in [n]} \left\| \sum_{i \in [m]} M_{ij}^*M_{ij'} \right\|_2^2.$$
Using Claim~\ref{cla-CS-block} and the fact that the supports of two columns of $M$ intersect in at most $t$ blocks, we continue:
$$S \leq  t \sum_{i \in [m]}  \sum_{j \neq j' \in [n]} \left\|M_{ij}^*M_{ij'} \right\|_2^2. $$
Now, applying Claim~\ref{cla-rowest} and using the fact that each row of $M$ has at most $q$ non-zero blocks, we get
$$ S \leq tmr(1 - 1/q). $$

We can now apply Lemma~\ref{lem-diagonaldom} with the above $L$ and $S$ to get that
\begin{eqnarray*}
	cn - \rank(H) &\leq& \frac{cnmtr(1-1/q)}{(mr/nc + o(1))^2(nc) + mtr(1-1/q)) }\\
	&=& \frac{cn}{1+X} + o(1),
\end{eqnarray*}
with $ X = \frac{mrq}{cnt(q-1)}$. Since this inequality holds for all $\eps$ we can take $\eps$ to zero and conclude that it holds without the $o(1)$ term as well. The final observation is that $\rank(M) = \rank(H)$ and so we are done.
\end{proof}

We can now prove the main rank theorem for design matrices.
\begin{proof}[Proof of Theorem~\ref{thm-rank-design}]
Let $A \in \M_{m,n}(r,c)$ be a $(q,k,t)$-design matrix. Let $B \in \M_{nk,n}(r,c)$ be the matrix given by Claim~\ref{cla-regularize}. So $B$ is a $(q,k,qt)$-design matrix in regular form	with $\rank(B) \leq \rank(A)$. By Claim~\ref{cla-regular-scalable} $B$ is scalable. Thus, we can apply Lemma~\ref{lem-design+scale=rank} to conclude that
$$\rank(A) \geq  \rank(B) \geq cn - \frac{cn}{1+X}, $$ with 
$$ X = \frac{nkrq}{cntq(q-1)} = \frac{kr}{ct(q-1)}.$$ This completes the proof.
\end{proof}

\section{Projective rigidity}\label{sec-rigidity}

Below, we will prove the following rigidity theorem (following some corollaries and preliminaries).
\begin{thm}[Rigidity theorem]\label{thm-rigidity-triples}
	Let $V = (v_1,\ldots,v_n) \in (\C^d)^n$ be a list of $n$ points in $\C^d$ and let $T \subset {[n] \choose 3}$ be a multiset of triples on the set $[n]$ so that all triples in $T$ are collinear in $V$. Suppose that $P_V$ is a non singular point of ${\cal K}_T$ (as required in the definition of $r$-rigidity) and that:
	\begin{enumerate}
	 \item For each $i \in [n]$ there are at least $k$ triples in $T$ containing $i$ (counting repetitions).
	 \item For every $i \neq j \in [n]$ there are at most $t$ triples in $T$ containing both $i$ and $j$ (counting repetitions).
	 \item For all $0< \ell < d$ there are at most $\frac{\ell}{d}k$  triples in $T$ (counting repetitions) so that all of them intersect at some point and the corresponding triples in $V$ are contained in an $\ell$-dimensional affine subspace. 	
	\end{enumerate}
	Then,  $(V,T)$ is $r$-rigid with $$ r = \left\lfloor \frac{2d^2tn}{2dt + k(d-1)} \right\rfloor.$$
\end{thm}

For example, if we have a triple system in $\C^2$ in which every pair is in exactly one triple and so that no line contains more than half the points, we get that the configuration is $15$-rigid. Indeed, setting $k=(n-1)/2, d=2,t=1$ the bound on $r$ becomes
$$ \left\lfloor  \frac{8n}{4 + (n-1)/2} \right\rfloor = \left\lfloor 16\cdot \frac{n}{n+7} \right\rfloor = 15.$$
We now discuss the implications for $\delta$-SG (Sylvester-Gallai) configurations, defined in \cite{BDWY12}.
\begin{define}[$\delta$-SG configuration]
A list $V = (v_1,\ldots,v_n) \in (\C^d)^n$ is called a $\delta$-SG configuration if for each $i \in [n]$ there exist at least $\delta (n-1)$ values of $j \in [n]\setminus\{i\}$ for which the line through $v_i,v_j$ contains a third point from the set.	
\end{define}

A theorem from \cite{DSW12} shows that a $\delta$-SG configuration must be contained in an affine  subspace of dimension at most $O(1/\delta)$.  We can use Theorem~\ref{thm-rigidity-triples} to prove the following result. In view of \cite{DSW12} this corollary is only interesting when $d = (1/\delta)$.

\begin{cor}
Let $V = (v_1,\ldots,v_n) \in (\C^d)^n$ be a $\delta$-SG configuration and let $T$ be the family of all collinear triples in $V$. Suppose that, for every $0<\ell< d$, any $\ell$-dimensional affine subspace of $\C^d$ contains at most $\frac{\delta \ell n}{d}$ points of $V$. Then $(V,T)$ is $\frac{12d}{\delta}$-rigid.
\end{cor}
\begin{proof}
 For each line containing $r\geq 3$ points we construct a triple multiset of $r^2 - r$ triples so that each point on the line is in exactly $3(r-1)$ triples and every pair is in at most $6$ triples (see Lemma~\ref{lem-steiner}). Taking the union of all these triples we get a family of	triples $T' \subset T$ (containment as sets, not multisets) and so it is enough to bound the rigidity of the pair $(V,T')$. Each point is in at least $k = 3\delta(n-1)$ triples in $T'$ and every pair is in at most $6$. To apply Theorem~\ref{thm-rigidity-triples} we need to argue that every $\ell$-dim affine subspace can contain at most $\frac{\ell}{d}k = \frac{3\ell \delta (n-1)}{d}$ intersecting triples in $T'$. If there exist an affine subspace $W$ that violates this inequality then $V$ must contain at least
$$  1+ 2 \cdot \frac{3\ell\delta(n-1)}{d} \cdot \frac{1}{6} > \frac{\delta \ell n}{d}$$ points of $V$ contradicting the assumptions. Applying Theorem~\ref{thm-rigidity-triples} (with $t=6$ and $k = \delta(n-1)$) we get that $(V,T')$ is $r$-rigid with 
$$ r = \left\lfloor \frac{2d^26n}{2d6 + \delta(n-1)(d-1)} \right\rfloor\leq \frac{12 d}{\delta}.$$ 
\end{proof}

\subsection{The rigidity matrix}
For a pair $(V,T) \in COL(n,d)$ we define a matrix $A = A(V,T) \in \M_{m,n}(d-1,d)$ with $m = |T|$ called the rigidity matrix of $(V,T)$. The matrix will be defined so that $dn - \rank(A)$ will upper bound the rigidity of $(V,T)$. To this end, we first define a certain $d-1 \times d$ block that will be used in the construction of $A$.

\begin{define}
Let $w = (w_1,\ldots,w_d) \in \C^d$ we define the matrix $\Delta(w) \in \M_{d-1,d}(1,1)$ as
\begin{equation*}
	\Delta(w) = \left(\begin{matrix}
  w_2 & -w_1 & 0    & & \cdots & 0 \\
  w_3 & 0    & -w_1 & 0 & \cdots & 0\\
   \cdots & & & & & \\
   w_d & 0 & & \cdots & 0 & -w_1 
\end{matrix}\right).
\end{equation*}	
Notice that, if $w_1 \neq 0$, then $\ker(\Delta(w)) = \spn(w)$.
\end{define}

\begin{define}[rigidity matrix]
Given $(V,T) \in COL(n,d)$ we construct $A = A(V,T) \in \M_{m,n}(d-1,d)$ with $m = |T|$ as follows: For each triple $(i,j,k) \in T$ we add to $A$ a row that has entry $\Delta(v_j - v_k)$ in position $i$, entry $\Delta(v_k-v_i)$ in position $j$, entry $\Delta(v_i - v_j)$ in position $k$ and zero blocks everywhere else. If $T$ is a multiset and a triple repeats several times, we also repeat the corresponding row in $A$ the same number of times.
\end{define}

\begin{claim}\label{cla-rigidity-matrix}
	If $A(V,T)$ has rank $dn - r$ then $(V,T)$ is $r$-rigid.
\end{claim}
\begin{proof}
Let $P(t)$ be a smooth curve in $\cK_T \subset \C^{nd}$ with $P(0) = P_V$. Let $\dot{P}(t)$ be the tangent vector. Then we claim that $A \cdot \dot{P}(0) = 0.$ By the construction of $A$ it is enough to show that, for a triple $(i,j,k) \in T$ we have
$$ \Delta(v_j - v_k)\cdot \dot{v_i}(0) + \Delta(v_k - v_i)\cdot \dot{v_j}(0) + \Delta(v_i - v_j)\cdot \dot{v_k}(0) = 0. $$

 This follows by taking the derivative w.r.t the variable $t$ of the $d-1$ identities (for $\ell = 2 \ldots d$) that hold for every collinear triple $v_i,v_j,v_k$ and any $t$.
\begin{equation}
	\det\left(\begin{matrix}
	1 & v_{i1}(t) & v_{i\ell}(t) \\
	1 & v_{j1}(t) & v_{j\ell}(t) \\
	1 & v_{k1}(t) & v_{k\ell}(t) 
\end{matrix}	
 \right) = 0.
\end{equation}
 Hence, the vector $\dot{P}(0)$ must lie in an $r$ dimensional subspace. This implies that the dimension of $\cK_T$  at $P_V$ is at most $r$.
\end{proof}

\subsection{Proof of Theorem~\ref{thm-rigidity-triples}}

Let $(V,T)$ be as in the statement of the theorem and let $A = A(V,T)$ be the corresponding rigidity matrix. We may assume w.l.o.g that the vectors $v_1,\ldots,v_n$ forming $V$ are distinct in the first coordinate (this can be achieved by applying a generic affine transformation).

\begin{claim}\label{cla-delta-wellspread}
Let $w_1,\ldots,w_k \in \C^d$ be such that the first coordinate in each $w_i$ is non zero and such that, for all $0< \ell <d$, any $\ell$-dimensional subspace of $\C^d$ contains at most $\frac{\ell}{d}k$ of the $w_i$'s. Then, the set of matrices $\Delta(w_1),\ldots,\Delta(w_k) \in \M_{d-1,d}(1,1)$ is well-spread.	
\end{claim}
\begin{proof}
Fix a subspace $V \subset \C^d$ of dimension $0<\ell<d$. We have that $\dim \left(\Delta(w_i)(V) \right)$ is equal to $\ell-1$ if $w_i \in V$ and to $\ell$ otherwise. Hence,
\begin{eqnarray*}
	\sum_{i \in [k]} \dim(\Delta(w_i)(V) &\geq& (k\ell/d)(\ell - 1) + (k - (k\ell)/d)\ell \\
	&=& \frac{k\ell(d-1)}{d}.
\end{eqnarray*}	
We then only have to argue that the definition of well-spread set is satisfied also for the special case of $V = \{0\}$ and $V = \C^d$. The  first is trivial to see and the second follows since each $\Delta(w_i)$ is full rank.
\end{proof}

\begin{claim}\label{cla-rigidity-design}
	The rigidity matrix $A \in \M_{m,n}(d-1,d)$ is a $(3,k,t)$-design matrix.
\end{claim}
\begin{proof}
By construction, each row of $A$ has three non zero blocks. Pairwise intersections of columns follow from the assumption that at most $t$ triples contain a particular pair of points. Now, consider $k$ triples of $T$ containing a particular point $v_i$ (we assume at least $k$ such triples exist). The corresponding blocks in the $i$'th column of $A$ are given by $\Delta(v_k - v_j)$ with $v_j,v_k$ being the other two points in that triple. Notice that all the vectors $v_k-v_j$ have a non-zero first coordinate and so we can use the fact that the kernel of  $\Delta(v_k-v_j)$  is $\spn(v_k - v_j)$. Since we assume that no $\ell$-dimensional affine subspace contains more than $\frac{\ell}{d}k$ of these $k$ (intersecting) triples,  by Claim~\ref{cla-delta-wellspread},  these $k$ entries will form a well-spread set. 
\end{proof}

Using the last claim, we can apply Theorem~\ref{thm-rank-design} to conclude that $$ dn - \rank(A) \leq \frac{dn}{1+\frac{k(d-1)}{2dt}} = \frac{2d^2tn}{2dt + k(d-1)}.$$ Noticing that the rank is an integer, we can add the floor to the obtained bound. This completes the proof of the theorem \qed


\section{Sylvester-Gallai for subspaces}\label{sec-highsg}

In this section we prove Theorem~\ref{thm-tightsg}. Let $k = \delta (n-1)$ and assume w.l.o.g that $k$ is an integer. For each $i \in [n]$ pick some basis $B_i = \{v_{i1},\ldots,v_{i\ell}\}$ for the subspace $V_i$. Let $A_V \in \M_{n\ell,d}(1,1)$ be the matrix whose first $\ell$ rows are the elements of $B_1$, the next $\ell$ rows are the elements of $B_2$ etc up to $B_n$. Our goal is then to prove an upper bound on the rank of $A_V$. For that purpose we will construct another matrix $A_C \in \M_{m,n}(\ell,\ell)$ of high rank such that $A_C \cdot A_V = 0$. 

We will now describe how to construct the matrix $A_C$. The first step is to construct a multiset of triples $T \subset {[n] \choose 3}$. We will use the following simple lemma from \cite{DSW12}.
\begin{lem} \label{lem-steiner}
Let $r\geq 3$.  Then there exists a multiset $U\subset {[r] \choose 3}$ of $r^2-r$ triples satisfying the following properties:
\begin{enumerate}
\item For each $i\in [r]$ there are exactly $3(r-1)$ triples in $U$ containing $i$ as an element.
\item  For every pair $i,j\in [r]$ of distinct elements there are at most 6 triples in $U$ containing both i and j as elements.
\end{enumerate}
\end{lem}

Notice that we are using multisets as, for example, if $r=3$ we must use the same (and only) triple with multiplicity $6$. Since the pair-wise intersections of the $V_i$'s are all trivial, every pair of them spans a $2\ell$ dimensional subspace of $\C^d$. We will call a $2\ell$ dimensional subspace of $\C^d$ {\em special} if it contains at least three of the $V_i$'s. For every special $2\ell$-dimensional space containing $r \geq 3$ spaces among the $V_i$'s we use Lemma~\ref{lem-steiner} to construct a multiset of $r^2 - r$ triples on the $r$ spaces contained in that special subspace satisfying the  two conditions of the lemma (we view these triples as triples in $[n]$ since each subspace is indexed by an element of $[n]$). We then define the triple multiset $T \subset {[n] \choose 3}$ to be the union (counting multiplicities) of all triples obtained this way (going over all special $2\ell$-dimensional spaces). 

\begin{claim}
	The triple multiset  $T \subset {[n] \choose 3}$ constructed above satisfies the following three conditions (counting multiplicities).
	\begin{itemize}
		\item If $\{i,j,k\} \in T$ then $V_k \subset V_i + V_j$.
		\item Each $i \in [n]$ appears in at least $3k$ triples in $T$.
		\item Every pair $i \neq j$ appears together in at most $6$ triples in $T$.
	\end{itemize}
\end{claim}
\begin{proof}
	The first item is satisfied since we only take triples contained in a $2\ell$ dimensional space and every pair has trivial intersection (and so spans the entire $2\ell$-dimensional space). To prove the second item, fix some $i \in [n]$ and suppose $V_i$ is contained in $s$ special $2\ell$-dimensional spaces $W_1,\ldots,W_s$  such that $W_j$ contains $r_j \geq 3$ spaces among the $V_1,\ldots,V_n$ (including $V_i$). By the conditions of the theorem, we know that $\sum_{j=1}^s (r_i-1) \geq k$. Hence, using the bounds from Lemma~\ref{lem-steiner} $V_i$ (or actually $i$) will be in $\sum_{j=1}^s 3(r_i-1) \geq 3k$ triples in $T$. The last item follows from the fact that a particular pair $V_i,V_j$ can belong to at most one special $2\ell$-dimensional space and then using the bound on pairs from Lemma~\ref{lem-steiner}.
\end{proof}

We now construct the matrix $A_C \in \M_{m,n}(\ell,\ell)$ by adding to $A_C$ a specially constructed row (of $\ell \times \ell$ blocks) for each triple in $T$ (if a triple repeats more than once we  also repeat the corresponding row the same number of times). The construction of the row is given in the following claim.
\begin{claim}\label{cla-blockrow}
	Let $t = \{i_1,i_2,i_3\} \in T$, then there exists a row matrix $R^{(t)} \in \M_{1,n}(\ell,\ell)$ with the following properties.
	\begin{enumerate}
		\item For each $i \not\in \{i_1,i_2,i_3\}$, the $i$'th block in $R^{(t)}$ is zero.
		\item The three blocks of $R^{(t)}$ indexed by $i_1,i_2,i_3$ are non singular $\ell \times \ell$ matrices.
		\item The product $R^{(t)} \cdot A_V$ is zero (viewed as an $\ell \times d$ scalar matrix).
	\end{enumerate}
\end{claim}
\begin{proof}
	Since $V_{i_1},V_{i_2},V_{i_3}$ are all contained in a $2\ell$ dimensional space (spanned by any two of them), every basis element in one of the spaces, say in $V_{i_1}$, is spanned by the basis elements in the other two. Let $B_i$ denote the matrix whose rows are the elements of the basis of $V_i$. We can thus find $\ell \times \ell$ matrices $C_2,C_3$ so that $$ B_{i_1} = C_2 \cdot B_{i_2} + C_3 B_{i_3}.$$ Moreover, both matrices $C_2,C_3$ are non singular, since otherwise  $V_{i_1}$ would intersect one of the spaces $V_{i_2},V_{i_3}$ non-trivialy. Hence, we can take the row $R^{(t)}$ to have the identity $\ell \times \ell$ block in position $i_1$ and the non singular blocks $-C_2,-C_3$ in positions $i_2,i_3$ (with zeros everywhere else). By construction of $A_V$ we have that the product $R^{(t)} \cdot A_V$ is zero.
\end{proof}

We now take the matrix $A_C \in \M_{m,n}(\ell,\ell)$ to have the rows (in whatever order we wish) $R^{(t)}$ for all $t \in T$ (counting multiplicities). By the last claim we have that $A_C \cdot A_V = 0$.

\begin{claim}
	The matrix $A_C$ is a $(3,3k,6)$-design matrix.
\end{claim}
\begin{proof}
	First notice that, by construction, each row of $A_C$ has at most three non zero blocks. By properties of the triple system $T$, every pair of columns $i \neq j$ will have at most $6$ rows of $A_C$ in which both columns are non zero (since there are at most 6 triples in $T$ containing both $i$ and $j$). So we only need to show that each column contains at least $3k$ blocks that form a well spread (multi)set.	By Claim~\ref{cla-blockrow}, each non zero block in $A_C$ is non singular and so, by Comment~\ref{com-nonsingular} , it is enough to show that each column contains at least $3k$ non zero blocks. This follows from the properties of $T$ since each $i$ appears in at least $3k$ triples.
\end{proof}

We now  apply Theorem~\ref{thm-rank-design} to bound the rank of $A_C$:
\[ \rank(A_C) \geq \ell n - \frac{\ell n}{1 + k/4}. \]
Using the identity $A_C \cdot A_V = 0$ we conclude that 
\[ \rank(A_V) \leq \frac{4\ell n}{k+4} < \frac{4\ell}{\delta}. \]  Now, using the fact that the rank is an integer and that we have a strict inequality we can in fact bound the rank by $\lceil 4\ell/\delta \rceil - 1$. This concludes the proof of Theorem~\ref{thm-tightsg}. \qed

\section{Incidences between lines and curves}\label{sec-incidence}

In this section we use Theorem~\ref{thm-rank-design} to prove bounds on the incidence structure of arrangements of lines and curves in $\C^d$.  We begin by restating our theorem handling intersections of lines.

\begin{thm}\label{thm-lineincidence}
	Let $L_1,\ldots,L_n \subset \C^d$ be distinct lines such that each $L_i$ intersects at least $k$ other lines and, among those $k$ lines, at most $k/2$ have the same intersection point on $L_i$. Then, the $n$ lines are contained in an affine subspace of dimension at most $\left\lfloor \frac{4n}{k+2}\right\rfloor -1$.
\end{thm}

This theorem can be equivalently stated as the following statement about two dimensional subspaces. 

\begin{thm}\label{thm-homincidence}
	Let $V_1,\ldots,V_n \subset \C^d$ be distinct two dimensional subspaces such that each $V_i$ non-trivially intersects at least $k$ other $V_j$'s and, among those $k$ subspaces, at most $k/2$ have the same intersection with  $V_i$. Then $$ \dim(V_1 + \cdots + V_n) \leq \left\lfloor \frac{4n}{k+2}\right\rfloor.$$
\end{thm}
\begin{proof}[Proof of equivalence of Theorem~\ref{thm-lineincidence} and Theorem~\ref{thm-homincidence}]

Suppose Theorem~\ref{thm-lineincidence} holds and proceed to prove Theorem~\ref{thm-homincidence} as follows. Let $H$ be a generic affine hyperplane (not passing through the origin) and let $L_i = V_i \cap H$ be the set of $n$ lines obtained by intersecting each $V_i$ with $H$. Clearly, the incidence structure remains the same and so we can apply Theorem~\ref{thm-lineincidence} to claim that the lines $L_1,\ldots,L_n$ are contained in an affine subspace (inside $H$) of dimension at most  $\left\lfloor \frac{4n}{k+2}\right\rfloor - 1$. This results in a dimension bound of $\left\lfloor \frac{4n}{k+2}\right\rfloor$ on the $V_i$'s since we add back the origin.
	
	In the opposite direction, suppose Theorem~\ref{thm-homincidence} holds and proceed to prove Theorem~\ref{thm-lineincidence} as follows. Let $L_1,\ldots,L_n \subset \C^d$ be lines as in the theorem. Embed $\C^d$ into $\C^{d+1}$ as the hyperplane $x_{d+1}=1$. Each line $L_i$ defines a two dimensional subspace in $\C^{d+1}$ by taking its linear span. If the lines $L_i$ span a $d'$-dimensional affine subspace in $\C^d$ then the resulting arrangement of two dimensional spaces in $\C^{d+1}$ spans a $d'+1$ dimensional linear subspace. Again, the incidence structure stays the same and so we can apply Theorem~\ref{thm-homincidence} and subtract one from the resulting dimension bound.
\end{proof}

\subsection{Proof of Theorem~\ref{thm-homincidence}}

The overall proof structure is similar to the proof of Theorem~\ref{thm-tightsg}. We pick a basis $\{u_i,v_i\} \in \C^d$ for each $V_i$ and consider the $2n \times d$ (scalar) matrix $A_V$ whose rows are $u_1,v_1,u_2,v_2,\ldots,u_n,v_n$. To upper bound the rank of $A_V$ we will construct a matrix $A_C \in \M_{m,n}(1,2)$ of high rank such that $A_C \cdot A_V = 0$. As before, each row of $A_C$ will come from some dependency (in this case pair-wise intersection) among the spaces $V_1,\ldots,V_n$. More specifically, for every pair $V_i,V_j$ with non trivial intersection we add a row $R \in \M_{1,n}(1,2)$ to $A_C$ (rows can be added in whatever order we wish), where $R$ is constructed as follows. Let $a_1,b_1,a_2,b_2 \in \C$ be such that $a_1u_i + b_1v_i + a_2u_j + b_2v_j = 0$ and with $|a_1| + |b_1| \neq 0$ and $|a_2| + |b_2| \neq 0$ (such coefficients exist since there is non trivial intersection). We take the row $R$ to have the block $(a_1,b_1)$ in position $i$ and the block $(a_2,b_2)$ in position $j$, with zeros everywhere else. By construction we have $R \cdot A_V = 0$ and so we end up with $A_C \cdot A_V = 0$ as well. 

\begin{claim}
	The matrix $A_C$ constructed above is a $(2,k,1)$-design matrix.
\end{claim} 
\begin{proof}
	Clearly every row has at most two non zero blocks and a pair of columns can have at most one row in which both are non zero (the row corresponding to their intersection, if one exists). So we only need to show that each column has $k$ blocks forming a well spread set. Fix some column $i$ and let $(a_1,b_1),\ldots,(a_k,b_k)$ be the $k$ blocks in the $i$'th column appearing in rows corresponding to the intersections of $V_i$ with $k$ subspaces $V_{j_1},\ldots,V_{j_k}$ of which at most $k/2$ have the same intersection with $V_i$. This last condition implies that, of the $k$ row vectors $(a_1,b_1),\ldots,(a_k,b_k)$, at most $k/2$ are pairwise linearly dependent. This implies that they satisfy the definition of well-spread blocks. Indeed, since the blocks are $1 \times 2$, we only need to consider one dimensional subspaces $U \subset \C^2$ in the definition of well-spread. For such a subspace, the linear map $\phi_i$ from $\C^2$ to $\C^1$ defined by a block $(a_i,b_i)$ will have a one dimensional image on $U$ if and only if $(a_i,b_i)$ is not in the orthogonal complement of $U$. Since at most $k/2$ of the $(a_i,b_i)$ can be in $U^\perp$ we get that  
	$$ \sum_{i\in [k]} \dim(\phi_i(U)) \geq \frac{k}{2} = \frac{k}{2}\dim(U), $$ as required.
\end{proof}

Applying Theorem~\ref{thm-rank-design} on $A_C$ we get that $\rank(A_C) \geq 2n - \frac{2n}{1 + k/2}.$ Hence, $ \rank(A_V) \leq \frac{4n}{k+2}. $ Since the rank is integer we get $$\rank(A_V) = \dim\left(\sum_{i\in [n]}V_i\right) \leq \left\lfloor \frac{4n}{k+2}\right\rfloor. $$ This completes the proof of the theorem. \qed

\subsection{Generalizing to curves}

Here we extend Theorem~\ref{thm-lineincidence} to handle curves of higher degree. For our methods to work we must require that the curves are given in parametric form as the image of a low degree polynomial map. 

\begin{define}
We say that $\gamma \subset \C^d$ is a degree $r$ parametric curve if there exists $d$ polynomials $\gamma_1,\ldots,\gamma_d \in \C[t]$ of degree at most $r$ each such that $$\gamma = \{ (\gamma_1(t),\ldots,\gamma_d(t)) \,|\, t \in \C \}$$ and at least one of the $\gamma_i$'s is a non constant polynomial.
\end{define}

It is easy to see that a parametric degree $r$ curve as defined above also has degree at most $r$ under the usual algebraic geometry definition of degree (intersecting it with a generic hyperplane, we get at most $r$ intersection points). A parametric curve as defined above is also an irreducible curve as it is the image of an irreducible curve under a polynomial map. Combining these two facts, and using Bezout's theorem (see e.g., \cite{Hartshorne77}) we can deduce the following.

\begin{claim}\label{cla-bezout}
	Let $\gamma \neq \gamma'$ be two degree $r$ parametric curves. Then $$ | \gamma \cap \gamma'| \leq r^2.$$
\end{claim} 

We now restate our theorem for curve arrangements.

\begin{thm}\label{thm-curveincidence}
	 Let $\gamma_1,\ldots,\gamma_n \subset \C^d$ be degree $r$ parametric curves such that each $\gamma_i$ intersects at least $k$ other curves and, among those $k$ curves, at most $k/2r$ have the same intersection point on $\gamma_i$. Then, the $n$ curves are contained in a subspace of dimension at most $ \frac{2(r+1)^4n}{k}$.
\end{thm}

\begin{proof}[Proof of Theorem~\ref{thm-curveincidence}]
We take the same general steps appearing in the proof of Theorem~\ref{thm-lineincidence}. First, for each curve $\gamma_i$, let $v_{i0},\ldots,v_{ir} \in \C^d$ be such that $$ \gamma_i = \left\{ \sum_{j=0}^r v_{ij} \cdot t^j \,:\, t\in \C \right\}. $$ In other words, $v_{ij}$ contains the coefficients of $t^j$ in the $d$ polynomials defining $\gamma$. Clearly, upper bounding the dimension of the span of the $v_{ij}$'s (over all $i$ and $j$) will give an upper bound for the dimension of the smallest subspace containing all of the curves. For that purpose, let $\Gamma$ be the $n(r+1) \times d$ matrix whose first $r+1$ rows are $v_{10},\ldots,v_{1r}$, second $r+1$ rows are $v_{20},\ldots,v_{2r}$ etc.  

We will now use the incidences between the curves to  construct a design matrix $A \in \M_{m,n}(1,r+1)$ so that $A\cdot V = 0$. Each intersection between a pair of curves will give one row in $A$ as follows. Suppose $\gamma_i$ intersects $\gamma_{i'}$ for some $i$ and $i'$. Let $t,t' \in \C$ be such that
$$ \sum_{j=0}^r v_{ij} \cdot t^j = \sum_{j=0}^r v_{i'j} \cdot (t')^j. $$
Then, we can add a row $R$ to the matrix $A$ such that the $i$'th block of $R$ is $(1,t,t^2,\ldots,t^r)$, the $(i')$'th block of $R$ is $(1,t',\ldots,(t')^r)$ and all other blocks are zero. By construction we have that $R \cdot \Gamma = 0$ and so, we will end up with a matrix $A$ such that $A \cdot \Gamma = 0$. 

All is left is to argue that $A$ is a design matrix. 

\begin{claim}
 The matrix $A \in \M_{m,n}(1,r+1)$ constructed above is a $(2,k',r^2)$-design matrix with $k' \geq k/2$.
\end{claim}
\begin{proof}
	By construction, each row of $A$ contains at most $2$ non-zero blocks. By Bezout's theorem (Claim~\ref{cla-bezout}), two curves can intersect in at most $r^2$ points and so two columns of $A$ can have at most $r^2$ non zero common indices. To complete the proof we need to show that each columns of $A$ contains at least $k/2$ blocks forming a well spread (multi)set. Fix some column $i$, and notice that there are at least $k$ non-zero blocks in that column, each corresponding to an intersection of $\gamma_i$ with some other curve. Let $t_1,\ldots,t_k \in \C$ be such that the $k$ non-zero blocks in the $i$'th column are given by $(1,t_j,\ldots,t_j^r)$ with $j=1\ldots k$. Some of the $t_i$'s could be the same (if a single point on $\gamma_i$ is the intersection point with more than one curve). Notice that, by Vandermonde's theorem, if we take $r+1$ distinct values of $t_i$ then the corresponding blocks (treated as row vectors in $\C^{r+1}$ are linearly independent and thus form a basis of $\C^{r+1}$. Our strategy for picking a large well-spread set among these $k$ block is as follows: We will greedily pick $r+1$ blocks corresponding to $r+1$ {\em distinct} intersection points and add them to our set. As long as we can find $r+1$ distinct intersections we continue. If we can't find such a set, it means that all the remaining intersection points on $\gamma_i$ are concentrated in at most $r$ points. Since each point can intersect at most $k/2r$ curves from the original $k$ (per the conditions of the theorem), there could be at most $(k/2r)\cdot r = k/2$ points left. This means that we managed to construct a (multi)set of  $k' \geq k/2$ blocks in a way that there is a partition of them into $k'/(r+1)$ linearly independent sets, each of size $r+1$. It is now easy to see that such a set is well-spread since a subspace $V \subset \C^{r+1}$ of dimension $\ell$ can contain at most $\frac{k'\ell}{r+1}$ of the $k'$ blocks (at most $\ell$ from each of the linearly independent sets in the partition). 
\end{proof}

To finish the proof we apply Theorem~\ref{thm-rank-design} to conclude that
$$ \rank(A) \geq (r+1)n - \frac{(r+1)n}{1 + \frac{k}{2(r+1)r^2}}. $$
This implies that 
$$ \rank(\Gamma) \leq \frac{(r+1)n}{1 + \frac{k}{2(r+1)r^2}} \leq \frac{2(r+1)^4n}{k}. $$ This completes the proof.
\end{proof}

%
%
%
%
%


\bibliographystyle{amsplain}


\begin{dajauthors}
\begin{authorinfo}[dvir]
  Zeev Dvir\\
  Princeton University\\
  Princeton, NJ, USA\\
  zdvir\imageat{}princeton\imagedot{}edu \\
  \url{http://www.cs.princeton.edu/~zdvir/}
\end{authorinfo}
\begin{authorinfo}[garg]
  Ankit Garg\\
  Microsoft Research New England\\
  Cambridge, MA, USA\\
  garga\imageat{}microsoft\imagedot{}com \\
  \url{https://www.microsoft.com/en-us/research/people/garga/}
\end{authorinfo}
\begin{authorinfo}[oliveira]
  Rafael Oliveira \\
  University of Toronto\\
  Toronto, Canada\\
  rafael\imageat{}cs\imagedot{}toronto\imagedot{}edu\\
  \url{http://www.cs.utoronto.ca/~rafael/index.html}
\end{authorinfo}
\begin{authorinfo}[solymosi]
  J\'{o}zsef Solymosi\\
  University of British Columbia\\
  Vancouver, BC, Canada\\
  solymosi\imageat{}math\imagedot{}ubc\imagedot{}ca\\
  \url{http://www.math.ubc.ca/~solymosi/}
\end{authorinfo}
\end{dajauthors}

\end{document}